\documentclass{amsart}
\pdfoutput=1

\usepackage{amssymb}
\usepackage[usenames,dvipsnames]{xcolor} 
\usepackage[colorlinks, 
            linkcolor=MidnightBlue, 
            citecolor=MidnightBlue, 
            urlcolor =BlueViolet,
           ]{hyperref}
\usepackage{booktabs}
\usepackage{mathtools}
\usepackage[all]{xy}
\usepackage{pgf}
\usepackage{tikz}
\usepackage{colonequals}

\newcommand{\BC}{{\mathbf C}}
\newcommand{\BQ}{{\mathbf Q}}
\newcommand{\BR}{{\mathbf R}}
\newcommand{\BZ}{{\mathbf Z}}

\newcommand{\CF}{{\mathcal F}}
\newcommand{\CM}{{\mathcal M}}
\newcommand{\CO}{{\mathcal O}}

\newcommand{\Mu}{{\mathrm{M}}}
\newcommand{\maxl}{{\textup{max}}}
\newcommand{\alphahat}{\widehat{\alpha}}
\newcommand{\betahat}{\widehat{\beta}}

\DeclareMathOperator{\degree}{deg}
\DeclareMathOperator{\End}{End}
\DeclareMathOperator{\Hom}{Hom}
\DeclareMathOperator{\Jac}{Jac}
\DeclareMathOperator{\PSL}{PSL}
\DeclareMathOperator{\Trace}{Trace}

\newtheorem{theorem}{Theorem}
\newtheorem{lemma}[theorem]{Lemma}
\newtheorem{corollary}[theorem]{Corollary}
\newtheorem{proposition}[theorem]{Proposition}

\theoremstyle{remark}
\newtheorem{rem}[theorem]{Remark}
\newtheorem{example}[theorem]{Example}

\theoremstyle{definition}

\begin{document}

\title[Curves with maps of every degree]
      {Curves of genus two with maps of every degree to a fixed elliptic curve}

\author{Everett W. Howe}
\address{Independent mathematician, 
         San Diego, CA 92104, USA}
\email{\href{mailto:however@alumni.caltech.edu}{however@alumni.caltech.edu}}
\urladdr{\href{http://ewhowe.com}{http://ewhowe.com}}

\date{21 September 2026}




\begin{abstract}
We show that up to isomorphism there are exactly twenty pairs $(C,E)$, where $C$
is a genus-$2$ curve over $\BC$, where $E$ is an elliptic curve over~$\BC$, and
where for every integer $n>1$ there is a map of degree~$n$ from $C$ to~$E$. We
also show that for every genus-$2$ curve $C$, there is an integer $n$ with
$1 < n \le 59$ such that there is no minimal degree-$n$ map from $C$ to an
elliptic curve.
\end{abstract}

\maketitle

\section{Introduction}
\label{S-intro}
Curves of genus two that have nonconstant maps to elliptic curves have been 
studied for nearly $200$ years, beginning with work of Legendre in 1828. Below,
we will briefly review some of the work of the early researchers in the field
--- Legendre, Jacobi, Weierstrass, Kowalevski, Poincar\'e, Picard, Goursat,
Brioschi, and others --- but for now we will simply note that the problem we
consider in this paper is one that could have been understood by these authors,
with just a little tweaking of the terminology. Namely, we address the question
of whether there exists a genus-$2$ curve $C$ over the complex numbers~$\BC$, 
and an elliptic curve $E$ over $\BC$, such that for every $n>1$ there exists a
degree-$n$ morphism from $C$ to $E$. (``Is there a hyperelliptic integral that 
can be reduced, via transformations of every degree $n>1$, to expressions
involving the same elliptic integral?'')

Perhaps surprisingly, the answer is yes.

\begin{theorem}
\label{T:main}
Up to isomorphism, there are exactly twenty pairs $(C,E)$ such that
\begin{enumerate}
\item $C$ is a curve of genus $2$ over the complex numbers $\BC$\textup{;}
\item $E$ is an elliptic curve over $\BC$\textup{;} and
\item for every $n>1$ there is a map of degree $n$ from $C$ to~$E$.
\end{enumerate}
\end{theorem}

Suppose $(C,E)$ is one of these twenty pairs. If we choose a base point $P$ 
on~$C$, then the set of maps from $C$ to $E$ that take $P$ to the origin of $E$
correspond by duality to embeddings of $E$ into $\Jac C$. These embeddings (plus
the zero map) form a $\BZ$-lattice in the $\BQ$-vector space 
$\Hom(E,\Jac C)\otimes\BQ$, whose dimension is at most twice the rank of
$\End E$. In our examples, we obtain $\BZ$-modules of rank~$4$, and the degree 
function is a quadratic form on each such module. The twenty pairs give rise to
only four different quadratic forms on $\BZ^4$, up to isomorphism. These 
quadratic forms are
\begin{align*}
q_1 &\colonequals 2w^2 + 3x^2 + 3y^2 + 4z^2 + 2xy                   \\
q_2 &\colonequals 2w^2 + 2x^2 + 3y^2 + 3z^2 + 2wz + 2xy             \\
q_3 &\colonequals 2w^2 + 3x^2 + 3y^2 + 4z^2 + 2wx + 2wy + 2xz + 2yz \\
q_4 &\colonequals 2w^2 + 3x^2 + 4y^2 + 6z^2 - 2wx + 2wz + 2xy + 4yz,
\end{align*}
and so in the course of proving Theorem~\ref{T:main} we will need the following
result.

\begin{proposition}
\label{P:forms}
Each of the quaternary quadratic forms $q_1$, $q_2$, $q_3$, $q_4$ represents 
every integer greater than $1$.
\end{proposition}

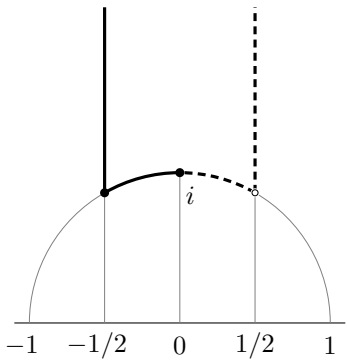
\begin{figure}[t]
\begin{tikzpicture}[scale=2]
\draw                             (-1.1   ,0     ) --   ( 1.1,0     );
\draw[very thin, gray]            (0      ,0     ) --   ( 0,1.0     );
\draw[very thick]                 (-0.5   ,0.8660) --   (-0.5,2.1   );
\draw[very thin, gray]            (-0.5   ,0     ) --   (-0.5,0.8660);
\draw[very thick, densely dashed] ( 0.5   ,0.8660) --   ( 0.5,2.1   );
\draw[very thin, gray]            ( 0.5   ,0     ) --   ( 0.5,0.8660);
\draw[very thick, densely dashed] ( 0.5   ,0.8660) arc ( 60: 90:1    cm);
\draw[very thick]                 ( 0     ,1     ) arc ( 90:120:1    cm);
\draw[very thin, gray]            ( 1     ,0     ) arc (  0: 60:1    cm);
\draw[very thin, gray]            (-1     ,0     ) arc (180:120:1    cm);

\draw (-1.05  ,-0.15) node {$-1$};
\draw (-0.55,-0.15) node {$-1/2$};
\draw (0,-0.15) node {$0$};
\draw (0.5,-0.15) node {$1/2$};
\draw (1  ,-0.15) node {$1$};
\draw (0.07 ,0.85) node {$i$};

\fill[black!100] (-0.5,0.866 ) circle (0.03cm);
\fill[black!100] ( 0  ,1     ) circle (0.03cm);
\fill[white!100] ( 0.5,0.866 ) circle (0.03cm);
\draw ( 0.5,0.866 ) circle (0.02cm);

\end{tikzpicture}
\caption{A strict fundamental domain $\CF_1$ for $\Gamma(1)$}
\label{F:X1}
\end{figure}

Let $\CF_1$ be the strict fundamental domain for $\Gamma(1)$ depicted in
Figure~\ref{F:X1}. Suppose $(C,E)$ is one of the twenty pairs from 
Theorem~\ref{T:main}, and let $\tau$ be the element of $\CF_1$ that corresponds
to~$E$.  We will show that $E$ has complex multiplication, so that $\tau$ is an
element of an imaginary quadratic field. We will also show that the curve $C$
has a period matrix of the form
\[
\begin{pmatrix}
1 & 0 & \tau/2 &      1/2 \\
0 & 1 &    1/2 & \sigma/2 
\end{pmatrix}
\]
where $\tau$ is as above and where $\sigma$ lies in the strict fundamental
domain $\CF_2$ for $\Gamma(2)$ depicted in Figure~\ref{F:X2}. (That there is a 
period matrix of this form, for any genus-$2$ curve with a map of degree $2$ to
an elliptic curve, is essentially a result of 
Picard~\cite{Picard1883, Picard1884}.) Table~\ref{table:data} gives the value of
$\tau$ and $\sigma$ for the each of the twenty pairs, along with the 
discriminants $\Delta_E$ and $\Delta_F$ of the endomorphism rings of $E$ 
and~$F$, and the quadratic form associated to the pair $(C,E)$. Pairs $(C,E)$ 
that have the same values of $\Delta_E$ and $\Delta_F$ can be obtained from one
another by Galois conjugation.

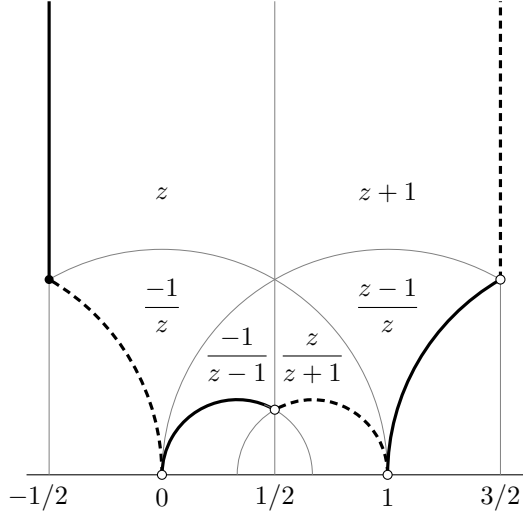
\begin{figure}[t]
\begin{tikzpicture}[scale=3]
\draw                             (-0.6   ,0     ) --   ( 1.6,0     );
\draw[very thick]                 (-0.5   ,0.8660) --   (-0.5,2.1   );
\draw[very thin, gray]            (-0.5   ,0     ) --   (-0.5,0.8660);
\draw[very thin, gray]            ( 0.5   ,0     ) --   ( 0.5,2.1   );
\draw[very thick, densely dashed] ( 1.5   ,0.8660) --   ( 1.5,2.1   );
\draw[very thin, gray]            ( 1.5   ,0     ) --   ( 1.5,0.8660);
\draw[very thick, densely dashed] ( 0     ,0     ) arc (  0: 60:1    cm);
\draw[very thin, gray]            ( 1     ,0     ) arc (  0:120:1    cm);
\draw[very thin, gray]            ( 1.5   ,0.8660) arc ( 60:180:1    cm);
\draw[very thick]                 ( 1.5   ,0.8660) arc (120:180:1    cm);
\draw[very thick]                 ( 0.5   ,0.2887) arc ( 60:180:0.333cm);
\draw[very thin, gray]            ( 0.6667,0     ) arc (  0: 69:0.333cm);
\draw[very thick, densely dashed] ( 1     ,0     ) arc (  0:120:0.333cm);
\draw[very thin, gray]            ( 0.5   ,0.2887) arc (120:180:0.333cm);

\draw (-0.55,-0.1) node {$-1/2$};
\draw (0,-0.1) node {$0$};
\draw (0.5,-0.1) node {$1/2$};
\draw (1,-0.1) node {$1$};
\draw (1.5,-0.1) node {$3/2$};

\fill[black!100] (-0.5,0.866 ) circle (0.02cm);
\fill[white!100] ( 0.5,0.2887) circle (0.02cm);
\draw ( 0.5,0.2887) circle (0.02cm);
\fill[white!100] ( 1.5,0.866 ) circle (0.02cm);
\draw ( 1.5,0.866 ) circle (0.02cm);
\fill[white!100] ( 0  ,0     ) circle (0.02cm);
\draw ( 0  ,0     ) circle (0.02cm);
\fill[white!100] ( 1  ,0     ) circle (0.02cm);
\draw ( 1  ,0     ) circle (0.02cm);

\draw (0.0,1.25) node {$z$};
\draw (0.0,0.75) node {$\displaystyle\frac{-1}{z}$};
\draw (0.333,0.53) node {$\displaystyle\frac{-1}{z-1}$};
\draw (0.666,0.53) node {$\displaystyle\frac{z\llap{\phantom{1}}}{z+1}$};
\draw (1.0,0.75) node {$\displaystyle\frac{z-1}{z}$};
\draw (1.0, 1.25) node {$z+1$};

\end{tikzpicture}
\caption{A strict fundamental domain $\CF_2$ for $\Gamma(2)$, whose closure is
tiled with images of the closure of the strict fundamental domain~$\CF_1$. The
tiles are labeled by the M\"obius transformation that takes $\CF_1$ to the given
tile. Note that $(3+\sqrt{-3})/2$ and $(3+\sqrt{-3})/6$ are not included in
$\CF_2$.}
\label{F:X2}
\end{figure}

\begin{table}[t]
\renewcommand\arraystretch{1.2}
\caption{Data for the twenty pairs $(C,E)$ from Theorem~\ref{T:main}}
\begin{center}
\begin{tabular}{r@{\qquad}rrccc}
\toprule
No. & $\Delta_E$ & $\Delta_F$ &  $\tau$              & $\sigma$               & Form  \\ 
\midrule
 1. & $-4$       & $-100$     & $       \sqrt{-1}   $ & $      5\sqrt{-1}    $ & $q_2$ \\ 
 2. &            &            &                       & $(12 + 5\sqrt{-1})/13$ &       \\
 3. & $-8$       & $-32$      & $       \sqrt{-2}   $ & $       \sqrt{-2} / 4$ & $q_1$ \\ 
 4. &            &            &                       & $( 4 +  \sqrt{-2})/ 4$ &       \\ 
 5. &            &            &                       & $( 1 +  \sqrt{-2})/ 2$ &       \\ 
 6. &            &            &                       & $( 2 +  \sqrt{-2})/ 4$ &       \\
 7. &            & $-72$      &                       & $( 6 +  \sqrt{-2})/ 6$ & $q_3$ \\ 
 8. &            &            &                       & $( 2 + 3\sqrt{-2})/ 2$ &       \\
 9. & $-12$      & $-3$       & $       \sqrt{-3}   $ & $(-1 +  \sqrt{-3})/ 2$ & $q_3$ \\ 
10. &            &            &                       & $( 1 +  \sqrt{-3})/ 2$ &       \\
11. & $-16$      & $-4$       & $      2\sqrt{-1}   $ & $       \sqrt{-1}    $ & $q_1$ \\ 
12. &            &            &                       & $  1 +  \sqrt{-1}    $ &       \\
13. & $-20$      & $-20$      & $       \sqrt{-5}   $ & $       \sqrt{-5}    $ & $q_2$ \\
14. &            &            & $(-1 +  \sqrt{-5})/2$ & $( 3 +  \sqrt{-5})/ 7$ &       \\
15. & $-24$      & $-24$      & $       \sqrt{-6}   $ & $( 2 +  \sqrt{-6})/ 2$ & $q_3$ \\
16. &            &            & $       \sqrt{-6} /2$ & $( 6 +  \sqrt{-6})/ 7$ &       \\
17. & $-36$      & $-36$      & $      3\sqrt{-1}   $ & $( 6 + 3\sqrt{-1})/ 5$ & $q_4$ \\ 
18. &            &            &                       & $( 4 + 3\sqrt{-1})/ 5$ &       \\  
19. &            &            & $(-1 + 3\sqrt{-1})/2$ & $  1 + 3\sqrt{-1}    $ &       \\ 
20. &            &            &                       & $( 3 +  \sqrt{-1})/ 3$ &       \\ 
\bottomrule
\end{tabular}
\label{table:data}
\end{center}
\end{table}

A map $\varphi$ from a curve $C$ to an elliptic curve $E$ is said to be
\emph{minimal} if it does not factor through an isogeny $F\to E$ of degree
greater than~$1$. We note that for our pairs $(C,E)$, for some values of $n$
there are no minimal maps $C\to E$ of degree~$n$. This follows from two more
general results that we prove in Section~\ref{S:intersection} by using work of 
Kani.

\begin{theorem}
\label{T:intersection1}
Let $C$ be a curve of genus~$2$ over $\BC$. Then for some $n$ in the set
$\{2,3,4,5,6,7,8,9,11,12,13,19,31,59\}$, there does not exist an elliptic curve
$E$ for which there exists a minimal map of degree $n$ from $C$ to~$E$.
\end{theorem}

\begin{theorem}
\label{T:intersection2}
Let $k$ be a positive integer and let 
\begin{align*}
N &= 18634{\,}63705{\,}23309{\,}79680\\
  &= 2^5\cdot 3^2\cdot 5\cdot 7\cdot 11\cdot 13\cdot 19\cdot 23\cdot 31\cdot 37\cdot 47\cdot 59\cdot 71\cdot 131\,.
\end{align*}
If $C$ is a curve of genus~$2$ over $\BC$, then for some $n$ in 
$\{2, 3, 4, kN\}$ there does not exist an elliptic curve $E$ for which there
exists a minimal map of degree $n$ from $C$ to~$E$.
\end{theorem}

The structure of this paper follows that of the proof of Theorem~\ref{T:main}.  
In Section~\ref{S:degree2} we recall some facts about genus-$2$ curves with
degree-$2$ maps to elliptic curves. In particular, the following proposition is
fundamental to our proof.

\begin{proposition}
\label{P:structure}
Suppose $C$ is a genus-$2$ curve with a degree-$2$ map $\varphi$ to an elliptic
curve $E$. Then there is a unique elliptic curve~$F$, a degree-$2$ map 
$\chi\colon C\to F$, and an isomorphism $\psi\colon E[2]\to F[2]$ such that the
kernel of $\varphi^*\times \chi^*\colon E\times F\to\Jac C$ is the graph of
$\psi$ and such that the following diagram commutes\textup:
\begin{equation}
\label{D:structure}
\vcenter{
\xymatrix{
E\times F \ar[rr]^{\left[\begin{smallmatrix}
\scriptstyle 1&\scriptstyle 0\\ \scriptstyle 0&\scriptstyle 1
\end{smallmatrix}\right]} &&
E\times F \ar[d]^{\varphi^*\times \chi^*} \\
\Jac C\ar[u]^{\varphi_*\times \chi_*}\ar[rr]^{2}&&\Jac C\rlap{\,.}
}}
\end{equation}
The pair $(\chi,\psi)$ is unique up to composition with automorphisms of~$F$.
Conversely, given two elliptic curves $E$ and $F$ and an isomorphism
$\psi\colon E[2]\to F[2]$, there is a genus-$2$ curve $C$ and a degree-$2$ map
$C\to E$ that gives rise to $F$ and $\psi$ as above, unless $\psi$ is the
restriction to $E[2]$ of an isomorphism $E\to F$, in which case there is no such
curve~$C$.
\end{proposition}

\begin{corollary}
\label{C:alphabeta}
Let notation be as in Proposition~\textup{\ref{P:structure}}, and let $\omega$ 
be a nonconstant map from $C$ to $E$. Let $\alpha$ be the endomorphism 
$\varphi_*\omega^*\colon E\to E$ and let $\beta$ be the morphism 
$\chi_*\omega^*\colon E\to F$. Then $\deg\omega = (\deg\alpha+\deg\beta)/2$, and for every 
$P\in E[2]$ we have $\beta(P) = \psi(\alpha(P))$.
\end{corollary}

There is a converse statement as well.

\begin{corollary}
\label{C:alphabeta2}
Let notation be as in Proposition~\textup{\ref{P:structure}}, and suppose 
$\alpha\colon E\to E$ and $\beta\colon E\to F$ are morphisms such that for every
$P\in E[2]$, we have $\beta(P) = \psi(\alpha(P))$. Then there is a nonconstant
map $\omega\colon C\to E$ with $\deg\omega = (\deg\alpha+\deg\beta)/2$ such that
$\alpha=\varphi_*\omega^*$ and $\beta = \chi_*\omega^*$.
\end{corollary}

Suppose $C$ is a genus-$2$ curve that has maps of every degree $n>1$ to an
elliptic curve~$E$. In Section~\ref{S:degrees3and4} we apply 
Corollary~\ref{C:alphabeta} to the degree-$3$ and degree-$4$ maps from $C$ 
to~$E$, and use the resulting information to deduce restrictions on the 
endomorphism ring of $E$ and on the relationship between $E$ and $F$. In 
particular, we prove the following proposition.

\begin{proposition}
\label{P:Deltap}
Suppose $C$ is a genus-$2$ curve that has maps of degree $2$, $3$, and $4$ to
an elliptic curve~$E$. Let $F$ be the curve associated as in 
Proposition~\textup{\ref{P:structure}} to a degree-$2$ map from $C$ to~$E$. Then
for one of the possibilities for $p$ and $\Delta$ listed below, the endomorphism
ring of $E$ has discriminant $\Delta$, and there is a cyclic isogeny from $E$ to
$F$ of degree~$p$.
\begin{enumerate}
\item $p = 1$ and $-\Delta\in\{3$, $4$, $7$, $11$, $12$, $16$, $19$, $20$, $24$, $27$, $28\}.$
\item $p = 2$ and $-\Delta\in\{4$, $7$, $8$, $12$, $15$, $16$, $20$, $23$, $24$, $31$, $36$, $39$, $40\}.$
\item $p = 3$ and $-\Delta\in\{3$, $4$, $8$, $11$, $12$, $16$, $19$, $20\}.$
\item \label{p=5}
      $p = 5$ and $-\Delta\in\{3$, $4$, $7$, $8$, $11$, $12$, $15$, $16$, $19$, $31$, $35$, $40$, $76$, $91$, $104$, $115$, $124$, $131$, $136$, $139$, $140\}$.
\end{enumerate}
\end{proposition}

After Proposition~\ref{P:Deltap}, we see that only finitely many pairs $(E,F)$
can occur. Suppose $(E,F)$ is a pair satisfying the conclusion of 
Proposition~\ref{P:Deltap} for some $p$ and $\Delta$, and suppose $\psi$ is one
of the six isomorphisms $E[2]\to F[2]$. We can compute the curve $C$ associated 
to this data as in the second statement of Proposition~\ref{P:structure}, if 
such a $C$ exists. For each such $C$ and $E$, we can compute a $\BZ$-basis for
$\Hom(C,E)$, and using Corollary~\ref{C:alphabeta2} we can compute the positive
definite quadratic form given by the degree map. It is then an easy matter to 
check whether this form represents all integers $n$ with $1<n<32$, which is 
obviously a necessary condition for the form to represent all integers $n>1$. 
Our method for doing this is explained in Section~\ref{S:enumerating}, and 
Magma~\cite{BosmaCannonEtAl1997} code for carrying out the computation is
available on the GitHub repository mentioned in 
Section~\ref{S:enumerating}.\footnote{We ran all of the code mentioned in this paper on an Apple M4 Max chip
running Magma V2.29-6 on MacOS Tahoe 26.2.}

It turns out that each quadratic form arising in this way that represents all 
the integers from $2$ to $31$ is equivalent to one of the forms $q_1$, $q_2$,
$q_3$, and $q_4$ given above, and therefore the $(C,E)$ pairs that we have found
satisfy the conditions of Theorem~\ref{T:main}. It is then a simple matter to
compute the data presented in Table~\ref{table:data}, and to see that there are
only $20$ such pairs.

In Section~\ref{S:models} we compute models for the curves~$C$. In 
Section~\ref{S:forms} we prove Proposition~\ref{P:forms}, and in 
Section~\ref{S:intersection} we prove Theorems~\ref{T:intersection1} 
and~\ref{T:intersection2}, and we compute the refined Humbert 
invariants~\cite{Kani2019} of the curves~$C$.

\begin{rem}
We note here that there is a result similar to Theorem~\ref{T:main} for fields
of positive characteristic, if we restrict our attention to \emph{ordinary} 
curves. Namely, if $C$ is an ordinary genus-$2$ curve over an algebraically
closed field $K$ of positive characteristic, and if $C$ has maps of every degree
$n>1$ to an elliptic curve $E$, then $C$ is the reduction of one of the curves
from Theorem~\ref{T:main}. This follows from Serre--Tate lifting 
(\cite{Katz1981},\cite{Norman1981}), which shows that an ordinary example over
$K$ can be lifted to an example in characteristic~$0$. Note, however, that in
general not all of the curves in the theorem will have good ordinary reduction
modulo a given prime, so there will not necessarily be $20$ examples of such
curves over a given~$K$.

We have not investigated the situation for non-ordinary curves over a field of
positive characteristic. The endomorphism ring of a supersingular elliptic curve
is a $\BZ$-module of rank $4$, so in some sense it should be easier for there to
exist maps $\alpha$ and $\beta$ as in Corollary~\ref{C:alphabeta2} that can 
produce an $\omega$ of a given degree. For this reason, we expect that over some
fields there will be examples of $(C,E)$ pairs that are not reductions of our
$20$ curves in characteristic zero.
\end{rem}

\begin{rem}
Here we give some historical background. As we mentioned at the beginning of
this section, the study of genus-$2$ curves with maps to elliptic curves goes
back nearly two centuries. In \S12 of the third supplement to his 
\emph{Trait\'e des fonctions elliptiques}~\cite{Legendre1828}, published
in~1828, Legendre shows how several ``ultra-elliptic'' integrals involving
expressions of the form $\sqrt{x(1-x^2)(1-k^2x^2)}$ can be expressed in terms of
elliptic integrals. Jacobi, in a postscript to his 1832 
review~\cite{Jacobi1832,Jacobi:Werke1} of Legendre's book, notes that Legendre's
examples can be generalized; rephrased in modern terminology, Jacobi's
observation is that every hyperelliptic curve of the form
\[ 
y^2 = x (x-1) (x-\lambda) (x-\mu) (x-\lambda\mu) 
\]
admits a degree-$2$ map to an elliptic curve. Legendre's examples come by 
taking $\lambda = -1$. Later, K\"onigsberger~\cite{Konigsberger1867} and 
Picard~\cite[\S9]{Picard1883} each proved that every genus-$2$ curve with a 
degree-$2$ map to an elliptic curve occurs in Jacobi's family.

The study of genus-$2$ curves with maps to elliptic curves continued, and
flourished, in the latter half of the $19$th century, with the focus shifting to
the period matrices of such curves and the endomorphism rings of their
Jacobians. In an 1874 paper, not published in a journal until 1884, 
Kowalevski~\cite{Kowalevski1884} quotes an unpublished result of Weierstrass 
that describes the period matrices of curves whose associated abelian integrals
can be reduced to elliptic integrals; in 1884 Poincar\'e~\cite{Poincare1884}
provided a proof of Weierstrass's theorem. For the special case of genus-$2$ 
curves, a better version of Weierstrass's result was given (independently) by
Picard~\cite{Picard1883}, and in 1884 Picard showed that his result can also be
deduced directly from that of Weierstrass~\cite{Picard1884}. At the very end of
the $19$th century, Humbert published a series of  
papers~\cite{Humbert1899,Humbert1900,Humbert1901} concerning genus-$2$ curves
whose Jacobians have endomorphism rings larger than $\BZ$; Humbert's curves
having ``singular relations with square invariant'' have  zero-divisors in their
endomorphism rings, and hence have maps to elliptic 
curves~\cite[Th\'eor\`eme~15, p.~247]{Humbert1899}.

Research in these matters has continued to this day. In more modern terminology,
one can fix an integer $n>1$ and study the moduli space of triples 
$(C,E,\varphi)$, where $\varphi\colon C\to E$ is a map of degree $n$ from a
curve of genus~$2$ to an elliptic curve. (Usually one demands in addition that 
the map be minimal, in the sense defined above.) Some work concerns the general
case (see for example~\cite{FreyKani1991,Kani1997,KaniSchanz1998}), but there is
also interest in considering specific small values of~$n$ and constructing more
or less explicit models of the corresponding moduli space, perhaps also giving
equations for the triples $(C,E,\varphi)$ themselves.

For $n=2$, Jacobi's previously-cited work gives such equations over
algebraically closed fields; in~\cite{HoweLeprevostPoonen2000}, the authors
analyze the situation over non-algebraically closed fields. For the case $n=3$,
there are works spanning 141 years, including
\cite{Bolza1898a,Bolza1898b,
Brioschi1891,
BrokerHoweLauterStevengagen2015,
Goursat1885a,
Goursat1885b,
Hermite1876,
Kuhn1988,
Shaska2004}.
The case $n=4$ is considered in older~\cite{Bolza1886} and more 
recent~\cite{BruinDoerksen2011} research, and there is work on the case $n=5$
as well~\cite{MagaardShaskaVolklein2009}. The paper~\cite{Kumar2015} considers
all $n$ up to $11$, but is more focused on models for the moduli space itself 
rather than on the triples $(C,E,\varphi)$, partly because the known models for
$C$ become quite complicated even for $n=4$.
\end{rem}

\subsubsection*{Acknowledgments}
This work was presented at ANTS XVII, and the author thanks the organizers of
the conference --- and also the participants --- for their roles in making the
conference a success. The author is grateful to Harun Kir for his comments and
suggestions on this paper, especially those concerning the computation in 
Section~\ref{S:intersection} of the refined Humbert invariants for the curves 
$C$ in Theorem~\ref{T:main}. 

\section{Consequences of the existence of a degree-\texorpdfstring{$2$}{2} map}
\label{S:degree2}

In this section we prove Proposition~\ref{P:structure} and its corollaries.

\begin{proof}[Proof of Proposition~\textup{\ref{P:structure}}]
Suppose $C$ is a genus-$2$ curve over $\BC$ with a degree-$2$ map $\varphi$ to
an elliptic curve~$E$. Then the special case $N=2$ 
of~\cite[Theorem~1.5]{Kani1997} shows that there is another elliptic curve $F$
and an isomorphism $\psi\colon  E[2]\to F[2]$ such that the Jacobian of $C$ is
isomorphic to the quotient of $E\times F$ by the graph of $\psi$, and such that
there is a degree-$2$ map $\chi\colon C\to F$.

Furthermore, if we let $G\subset (E\times F)[2]$ be the graph of $\psi$, then 
the isogeny $\varphi^*\times\chi^* \colon E\times F\to \Jac C$ has kernel $G$, 
and we have a diagram
\[
\xymatrix{
E\times F \ar[rr]^{\left[\begin{smallmatrix}
\scriptstyle 2&\scriptstyle 0\\ \scriptstyle 0&\scriptstyle 2
\end{smallmatrix}\right]} \ar[d]^{\varphi^*\times \chi^*} &&
E\times F  \\
\Jac C\ar[rr]^{1}&&\Jac C\ar[u]^{\varphi_*\times \chi_*}\rlap{\,.}
}
\]

We can then extend this diagram so that the compositions of the horizontal 
arrows on the top line and on the bottom line are the multiplication-by-$2$
maps:

\[
\xymatrix{
E\times F \ar[rr]^{\left[\begin{smallmatrix}
\scriptstyle 2&\scriptstyle 0\\ \scriptstyle 0&\scriptstyle 2
\end{smallmatrix}\right]} \ar[d]^{\varphi^*\times \chi^*} &&
E\times F \ar[rr]^{\left[\begin{smallmatrix}
\scriptstyle 1&\scriptstyle 0\\ \scriptstyle 0&\scriptstyle 1
\end{smallmatrix}\right]} &&
E\times F\ar[d]_{\varphi^*\times \chi^*}   \\
\Jac C\ar[rr]^{1}&&\Jac C\ar[u]^{\varphi_*\times \chi_*}\ar[rr]^{2}&&\Jac C\rlap{\,.}
}
\]

The right half of this diagram is nothing other than 
diagram~\eqref{D:structure}, which is what we want to show exists. The 
uniqueness of the pair $(\chi,\psi)$ up to automorphisms of~$F$ is part
of~\cite[Theorem~1.5]{Kani1997}, and the converse follows from this as well.
\end{proof}

\begin{rem}
\label{R:periodstructure}
We note that we can give a period matrix for the Jacobian of $C$ in terms of the
period matrices for $E$ and $F$ and the isomorphism $\psi$, as follows. First, 
$E$ has a period lattice $\Lambda_E$ of the form $\langle 1, \tau\rangle$ for a
unique $\tau$ in the fundamental domain $\CF_1$, and there is a unique $\sigma$
in the fundamental domain $\CF_2$ such that 
\begin{itemize}
\item $\Lambda_F \colonequals\langle 1, \sigma\rangle$ is a period matrix 
      for~$F$, and 
\item the isomorphism $\psi\colon E[2]\to F[2]$ sends the $2$-torsion point 
      $1/2 \bmod \Lambda_E$ of $E(\BC)$ to the $2$-torsion point 
      $\sigma/2 \bmod \Lambda_F$ of $F(\BC)$, and the point 
      $\tau/2\bmod \Lambda_E$ of $E(\BC)$ to the point 
      $1/2 \bmod \Lambda_F$ of $F(\BC)$.
\end{itemize}      
Then we can take
\begin{equation}
\label{EQ:periodmatrix}
\Lambda_C\colonequals
\begin{pmatrix}
1 & 0 & \tau/2 &      1/2 \\
0 & 1 & 1/2    & \sigma/2
\end{pmatrix}
\end{equation}
to be a period matrix for the Jacobian of~$C$. This is essentially a result of
Picard; see \cite{Picard1883} and \cite{Picard1884}.

We also know the sesquilinear form on $\BC^2$ that represents the principal
polarization on $\Jac C$, because it is derived from the product polarization
on $\Lambda_E\times \Lambda_F$.  Namely, if $\delta$ is any multiple of 
$\sqrt{-1}$ and we write $\tau = a + b\delta$ and $\sigma = c + d\delta$ for
real numbers $a,b,c,d$, then the sesquilinear form applied to elements 
$(z_1,z_2)$ and $(w_1,w_2)$ of $\BC^2$ gives the value
\begin{equation}
\label{EQ:polarization}
\Trace_{\BC/\BR} \Big( 
                 \frac{w_1\bar{z}_1}{b\delta} + \frac{w_2\bar{z}_2}{d\delta}
                 \Big)\,.
\end{equation}
One can check that the matrix of values of this pairing, applied to pairs of
column vectors in the basis for $\Lambda_C$ given above, is
\[
\begin{pmatrix*}[r]
0 & \phantom{-}0 & -1 &  0\\
0 &            0 &  0 & -1\\
1 &            0 &  0 &  0\\
0 &            1 &  0 &  0
\end{pmatrix*},
\]
so the pairing does indeed give a principal polarization on~$\Lambda_C$.
\end{rem}

\begin{proof}[Proof of Corollary~\textup{\ref{C:alphabeta}}]
Let $\alpha = \varphi_*\omega^*$ and $\beta = \chi_*\omega^*$, and let
$\alphahat$ and $\betahat$ be the dual morphisms of $\alpha$ and~$\beta$. We can
extend diagram~\eqref{D:structure} as follows:
\begin{equation}
\label{D:fundamental}
\vcenter{\xymatrix{
&&E\times F \ar[rr]^{\left[\begin{smallmatrix}
\scriptstyle 1&\scriptstyle 0\\ \scriptstyle 0&\scriptstyle 1
\end{smallmatrix}\right]} &&
E\times F \ar[d]_{\varphi^*\times \chi^*}
\ar[rrd]^{\alphahat + \betahat}&&\\
E\ar[rr]^{\omega^*}\ar[rru]^{\alpha\times\beta\quad}&&
\Jac C\ar[u]_{\varphi_*\times \chi_*}\ar[rr]^{2}
&&\Jac C\ar[rr]^{\omega_*} && E\rlap{\,.}
}}
\end{equation}
Following the bottom edge of the diagram gives us
multiplication by $2\deg\omega$ on $E$. The map from $E$ to $E$ we get from 
following the top edges of the diagram is the sum of the endomorphisms 
$\alphahat\alpha$ and $\betahat\beta$ of~$E$. But $\alphahat\alpha$ is
multiplication by $\deg\alpha$, and $\betahat\beta$ is multiplication
by~$\deg\beta$, so we see that $2\deg\omega = \deg\alpha+\deg\beta$, as claimed.

Let $P$ be a point of order $2$ on~$E$. Then the image of $P$ under the map from
the lower left of the diagram to the $E\times F$ on the upper right is the pair 
$(\alpha(P), \beta(P))$, while the image of $P$ in rightmost copy of $\Jac C$
is~$0$, because the middle map from $\Jac C$ to $\Jac C$ is multiplication 
by~$2$. Therefore, $(\alpha(P), \beta(P))$ lies in the kernel of the isogeny
$\varphi^*\times\chi^*$, which is the graph of~$\psi$, and it follows that
$\beta(P) = \psi(\alpha(P))$.
\end{proof}

\begin{proof}[Proof of Corollary~\textup{\ref{C:alphabeta2}}]
Given $\alpha$ and $\beta$ as in the statement of the corollary, consider the
following diagram:
\begin{equation}
\label{D:converse}
\vcenter{
\xymatrix{
&&E\times F \ar[rr]^{\left[\begin{smallmatrix}
\scriptstyle 1&\scriptstyle 0\\ \scriptstyle 0&\scriptstyle 1
\end{smallmatrix}\right]} &&
E\times F \ar[d]_{\varphi^*\times \chi^*}
\ar[rrd]^{\alphahat + \betahat}&&\\
E\ar[rru]^{\alpha\times\beta\quad}&&
\Jac C\ar[u]_{\varphi_*\times \chi_*}\ar[rr]^{2}
&&\Jac C&& E\rlap{\,.}
}}
\end{equation}
Our goal is to produce a morphism $\omega\colon C\to E$ that will allow us to
extend this diagram to diagram~\eqref{D:fundamental}.

By assumption, we have $\beta(P) = \psi(\alpha(P))$ for every $P\in E[2]$, so
the kernel of the map $\alphahat + \betahat$ from $E\times F$ to $E$ contains
the kernel of $\varphi^*\times \chi^*$. It follows that there is a map 
$\varpi\colon\Jac C\to E$ that we can use to complete the triangle on the 
right-hand side of~\eqref{D:converse}. (We note that this map is unique, because
$\varphi^*\times \chi^*$ is an isogeny.)

Choose an Abel--Jacobi embedding of $C$ into its Jacobian, and let $\omega$ be
the composition of this embedding with $\varpi$. Then we automatically have
$\varpi = \omega_*$, and by duality we find that $\omega^*\colon E\to\Jac C$
completes the triangle on the left-hand side of~\eqref{D:converse}. This gives
us~\eqref{D:fundamental}, and proves the corollary.
\end{proof}

\section{Consequences of the existence of maps of degree \texorpdfstring{$3$}{3} and \texorpdfstring{$4$}{4}}
\label{S:degrees3and4}
In this section we prove Proposition~\ref{P:Deltap}. We begin with a lemma that
records some facts about endomorphism rings of elliptic curves with noncyclic 
endomorphisms of small degree. 

\begin{lemma}
\label{L:discs}
Let $E$ be an elliptic curve over $\BC$ that has a cyclic endomorphism $\alpha$,
and let $\Delta$ be the discriminant of the endomorphism ring of $E$.
\begin{enumerate}
\item If $\deg\alpha=2$, then $-\Delta\in\{4$, $7$, $8\}$.
\item If $\deg\alpha=3$, then $-\Delta\in\{3$, $8$, $11$, $12\}$.
\item If $\deg\alpha=4$, then $-\Delta\in\{7$, $12$, $15$, $16\}$.
\item If $\deg\alpha=5$, then $-\Delta\in\{4$, $11$, $16$, $19$, $20\}$.
\item If $\deg\alpha=6$, then $-\Delta\in\{8$, $15$, $20$, $23$, $24\}$.
\item If $\deg\alpha=7$, then $-\Delta\in\{3$, $7$, $12$, $19$, $24$, $27$, $28\}$.
\item If $\deg\alpha=10$, then $-\Delta\in\{4$, $15$, $24$, $31$, $36$, $39$, $40\}$.
\item \label{deg35} If $\deg\alpha=35$, then $-\Delta\in\{19$, $31$, $35$, $40$, $59$, $76$, $91,104$, $115$, $124$, $131$, $136,$ $139$, $140\}$.
\end{enumerate}
\end{lemma}

\begin{proof}
Since $E$ has a cyclic endomorphisms of positive degree, its endomorphism ring
$\CO$ is an imaginary quadratic order, and
$\CO\cong\BZ[(\Delta + \sqrt\Delta)/2]$. If we write
$\alpha = x + y(\Delta + \sqrt\Delta)/2$ for integers $x$ and $y$, then $x$ and
$y$ must be coprime to one another (because $\alpha$ is cyclic), and we have
$\degree\alpha = x^2 + \Delta xy + y^2(\Delta^2-\Delta)/4$. Given a value for
$\degree\alpha$, it is a simple matter to find the discriminants $\Delta$ for
which it is possible to find coprime $x$ and $y$ satisfying the equality above.
We leave the details to the reader.
\end{proof}

\begin{proof}[Proof of Proposition~\textup{\ref{P:Deltap}}]
Suppose $C$ is a genus-$2$ curve over $\BC$ that has maps of degree $2$, $3$,
and $4$ to an elliptic curve~$E$. Let $\varphi$ be a degree-$2$ map from $C$ 
to~$E$, and let the elliptic curve $F$, the degree-$2$ map $\chi\colon C\to F$, 
and the isomorphism $\psi\colon E[2]\to F[2]$ be as in 
Proposition~\ref{P:structure}.

By Corollary~\ref{C:alphabeta}, the existence of the degree-$3$ map from $C$ to
$E$ implies that there is an endomorphism $\alpha$ of $E$ and a morphism 
$\beta\colon E\to F$ such that $\deg\alpha + \deg\beta = 6,$ and such that 
\begin{equation}
\label{EQ:antiisometry}
\beta(P)= \psi(\alpha(P)) \text{\quad for all $P\in E[2]$\,.}
\end{equation}
In particular,~\eqref{EQ:antiisometry} implies that 
\begin{equation}
\label{EQ:kernels}
\#(\ker\alpha)[2] = \#(\ker\beta)[2]\,.
\end{equation}
We enumerate the possibilities below. Note that Lemma~\ref{L:discs} tells us
the possible discriminants of the endomorphism ring of an elliptic curve with
a cyclic isogeny of certain degrees, and we use this without further comment in
the list below to indicate how each possibility is covered by one of the cases
in the statement of the proposition.
\begin{enumerate}
\item[1.] $\deg\alpha = 0$ and $\deg\beta = 6$. 
      This cannot happen, because $\#(\ker\alpha)[2] = 4$ while 
      $\#(\ker\beta)[2] = 2$, contradicting~\eqref{EQ:kernels}.
\item[2.] $\deg\alpha = 1$ and $\deg\beta = 5$. 
      This implies that $F$ is $5$-isogenous to $E$. We explore this case
      further in the discussion below.
\item[3.] $\deg\alpha = 2$ and $\deg\beta = 4$. 
      By~\eqref{EQ:kernels}, we see that $\beta$ must be a cyclic isogeny. More 
      specifically, \eqref{EQ:antiisometry} implies that $\ker\alpha$ is 
      contained in $\ker\beta$, so $\beta$ is the composition of $\alpha$ with a
      $2$-isogeny from $E$ to~$F$. This possibility therefore falls under the 
      case $p=2$ of the statement of the proposition.
\item[4.] $\deg\alpha = 3$ and $\deg\beta = 3$. 
      This falls under the case $p=3$ of the statement of the proposition.
\item[5.] $\deg\alpha = 4$ and $\deg\beta = 2$. 
      We see from~\eqref{EQ:kernels} that $\alpha$ must be a cyclic isogeny.
      Therefore this falls under the case $p=2$ of the statement of the
      proposition.
\item[6.] $\deg\alpha = 5$ and $\deg\beta = 1$. 
      This falls under the case $p=1$ of the statement of the proposition.
\item[7.] $\deg\alpha = 6$ and $\deg\beta = 0$. 
      Equation~\eqref{EQ:kernels} shows that this case cannot occur.
\end{enumerate}

The only possibility not covered by the conclusion of the proposition is that
$E$ is arbitrary and $F$ is $5$-isogenous to $E$. For the rest of the proof we 
will assume that we are in this case.

Now we consider the consequences of the existence of a degree-$4$ map from $C$
to~$E$. Corollary~\ref{C:alphabeta} implies that there is an endomorphism 
$\alpha$ of $E$ and a morphism $\beta\colon E\to F$ such that 
$\deg\alpha + \deg\beta = 8,$ with~\eqref{EQ:antiisometry} 
and~\eqref{EQ:kernels} holding. We list the possibilities, and again use 
Lemma~\ref{L:discs} without comment to show which cases of the proposition
covers them.
\begin{enumerate}
\item[1.] $\deg\alpha = 0$ and $\deg\beta = 8$. 
      From~\eqref{EQ:kernels} we see that $\ker\beta$ must contain $E[2]$, so
      $\beta$ is the composition of a $2$-isogeny $E\to F$ with the
      multiplication-by-$2$ map on $E$. We see that $F$ must be $2$-isogenous
      to~$E$. Since $F$ is also $5$-isogenous to $E$, we see that $E$ has an
      endomorphism of degree~$10$. We find that this possibility falls under the
      case $p=2$ of the proposition.
\item[2.] $\deg\alpha = 1$ and $\deg\beta = 7$. 
      We will discuss this case below.
\item[3.] $\deg\alpha = 2$ and $\deg\beta = 6$.
      This falls under the case $p=5$ of the statement of the proposition.
\item[4.] $\deg\alpha = 3$ and $\deg\beta = 5$. 
      This falls under the case $p=5$ of the statement of the proposition.
\item[5.] $\deg\alpha = 4$ and $\deg\beta = 4$. 
      If $\alpha$ is cyclic, then this falls under the case $p=5$ of the
      statement of the proposition. If $\alpha$ is not cyclic, then 
      by~\eqref{EQ:kernels} neither is $\beta$, which means that $F\cong E$.
      Therefore, there is an endomorphism of $E$ of degree~$5$. This falls under
      the case $p=1$ of the statement of the proposition.
\item[6.] $\deg\alpha = 5$ and $\deg\beta = 3$. 
      This falls under the case $p=3$ of the statement of the proposition.
\item[7.] $\deg\alpha = 6$ and $\deg\beta = 2$. 
      This falls under the case $p=5$ of the statement of the proposition.
\item[8.] $\deg\alpha = 7$ and $\deg\beta = 1$. 
      This falls under the case $p=1$ of the statement of the proposition.
\item[9.] $\deg\alpha = 8$ and $\deg\beta = 0$. 
      By~\eqref{EQ:kernels} we see that $\alpha$ cannot be cyclic, so it is the
      composition of  multiplication-by-$2$ with an endomorphism of $E$ of 
      degree~$2$. This falls under the case $p=5$ of the statement of the 
      proposition.
\end{enumerate}

This leaves us with one situation unaddressed: when there is both a $5$-isogeny 
from $E$ to $F$ and a $7$-isogeny from $E$ to~$F$, so that $E$ has an 
endomorphism of degree~$35$. If the endomorphism ring of $E$ does not have 
discriminant $-59$, then Lemma~\ref{L:discs} shows that this situation falls
under the case $p=5$ of the proposition. To finish the proof of the proposition,
we must show that the case of discriminant $-59$ cannot occur.

Suppose, in the situation of the proposition, that the endomorphism ring of the
elliptic curve $E$ has discriminant $-59$ and that $F$ is both $5$-isogenous and
$7$-isogenous to $E$. Since there are isogenies from $E$ to $F$ of coprime 
degrees, $F$ must also have endomorphism ring with discriminant $-59$. We check
that then there are morphisms from $E$ to $F$ of degrees $0$, $3$, $5$, and~$7$,
and no other degrees less than $9$, and that there are endomorphisms of $E$ of 
degrees $0$, $1$, and $4$, and no other degrees less than $9$.

Since we are assuming that there is a degree-$3$ map from $C$ to $E$, 
Corollary~\ref{C:alphabeta} says that there is an $\alpha_3\in \End E$ and a
$\beta_3\in\Hom(E,F)$ such that $\deg\alpha_3 + \deg\beta_3 = 6$ and such that
$\beta_3(P) = \psi(\alpha_3(P))$ for all $P\in E[2]$. The only possibility is
that $\deg\alpha_3 = 1$ and $\deg\beta_3 = 5$. Note that $\alpha_3 = \pm 1$, so
in fact $\beta_3(P) = \psi(P)$ for all $P\in E[2]$.

Likewise, since there is a degree-$4$ map from $C$ to $E$ there is an
$\alpha_4\in \End E$ and a $\beta_4\in\Hom(E,F)$ such that
$\deg\alpha_4 + \deg\beta_4 = 8$ and such that $\beta_4(P) = \psi(\alpha_4(P))$
for all $P\in E[2]$. The only possibility is $\deg\alpha_4 = 1$ and 
$\deg\beta_4 = 7$, so that $\alpha_4 = \pm 1$ and  $\beta_4(P) = \psi(P)$ for 
all $P\in E[2]$.

If we let $\betahat_4$ be the dual isogeny of $\beta_4$, so that 
$\betahat_4\beta_4 = 7$, then $P = \betahat_4(\psi(P))$ for all
$P\in E[2]$. Therefore, if we set $\gamma = \betahat_4\beta_3$, then 
$\gamma$ is a degree-$35$ endomorphism of $E$ that acts trivially on~$E[2]$.
That implies that $\gamma-1$ kills $E[2]$, so $\gamma-1 = 2\delta$ for some 
$\delta\in\End E$.

But we check that the only elements of norm $35$ in 
$\End E \cong \BZ[\frac{1+\sqrt{-59}}{2}]$ are $\frac{\pm 9 \pm \sqrt{-59}}{2}$,
and none of these can be written as 
$1 + 2\delta$ for $\delta \in \BZ[\frac{1+\sqrt{-59}}{2}]$. Therefore, $\End E$
cannot have discriminant~$-59$.
\end{proof}

\begin{rem}
We choose to exclude the discriminant $-59$ from the statement of 
Proposition~\ref{P:Deltap}, even though its exclusion complicates the proof,
because including it would make one of our later computational steps slightly
more awkward. See Remark~\ref{R:59}.
\end{rem}

\section{Enumerating possible examples}
\label{S:enumerating}

Proposition~\ref{P:Deltap} gives our first step toward our proof of 
Theorem~\ref{T:main} by specifying a finite list of possible pairs $(E,F)$ from
which to construct examples of pairs $(C,E)$ as in Theorem~\ref{T:main}. In this
section we explain how a computer calculation gives our second step toward the
proof, by greatly reducing the number of possibilities.

\begin{proposition}
\label{P:Deltas}
Let $C$ and $E$ be as in Theorem~\textup{\ref{T:main}}, let 
$\varphi\colon C\to E$ be a degree-$2$ map, let $F$ be as in
Proposition~\textup{\ref{P:structure}}, and let $\Delta_E$ and $\Delta_F$ be the
discriminants of the endomorphism rings of $E$ and~$F$, respectively. Then the 
pair $(\Delta_E, \Delta_F)$ is one of the following\textup{:}
\begin{align*}
& (-3,-3)   && (-7,-7)  && (-8,-72)  && (-12,-12) && (-16,-16) && (-20,-20) \\
& (-4,-4)   && (-8,-8)  && (-11,-11) && (-12,-48) && (-16,-64) && (-24,-24) \\
& (-4,-100) && (-8,-32) && (-12,-3)  && (-16,-4)  && (-19,-19) && (-36,-36). 
\end{align*}
In the cases where $\End E$ and $\End F$ have the same discriminant~$\Delta$, 
the curves $E$ and $F$ are isomorphic to one another when 
$-\Delta\in\{3,4,7,8,11,12,16,19,20\}$ and are not isomorphic to one another
when $-\Delta\in \{24,36\}$.
\end{proposition}

\begin{proof}
Our proof is computational, and Magma programs for carrying out the computation
are available at \url{https://github.com/everetthowe/many-maps}.

We narrow down our possibilities by using a weakened form of
Corollary~\ref{C:alphabeta}. Given two elliptic curves $E$ and~$F$, if there
exists an isomorphism $\psi\colon E[2]\to F[2]$ such that $E$, $F$, and $\psi$
are as in Proposition~\ref{P:structure}, then for every $n>1$ there exists an
endomorphism $\alpha_n$ of $E$ and a homomorphism $\beta_n\colon E\to F$ such 
that $2n = \deg\alpha_n + \deg\beta_n$, and such that~\eqref{EQ:antiisometry} 
holds. It follows that~\eqref{EQ:kernels} also holds, and this is the weaker
condition that we will check.

Given the $j$-invariants of $E$ and $F$, we can use the classical modular
polynomials $\Psi_m$ to determine whether there are endomorphisms of $E$, and 
homomorphisms $E\to F$, of any given (small) degree and with kernels containing
a given number of $2$-torsion points. The polynomial $\Psi_m$ has the property 
that there is a cyclic degree-$m$ isogeny between two elliptic curves with 
$j$-invariants $j_1$ and $j_2$ if and only if $\Psi_m(j_1,j_2) = 0$. Every 
isogeny $\beta$ can be factored into a cyclic isogeny composed with 
multiplication by a rational integer, and $\#(\ker\beta)[2]$ is determined by
the parity of the degree of the cyclic isogeny and the parity of the rational
integer.

To prove the proposition, we run through all $(\Delta,p)$ pairs listed in 
Proposition~\ref{P:Deltap}. For each $\Delta$, we construct the number field $K$
that contains the $j$-invariants of the elliptic curves $E$ whose endomorphism
rings have discriminant~$\Delta$; this is simply the number field defined by
the Hilbert class polynomial for $\Delta$, whose roots are precisely the
$j$-invariants in question. Then, for one such root $j_E$, we use the classical
modular polynomial $\Psi_p$ to construct an extension $L$ of $K$ that contains
the $j$-invariants $j_F$ of the elliptic curves $F$ that are $p$-isogenous
to~$E$.

For each such pair $(j_E, j_F)$, we use the method sketched above to compute the
set $S_E$ of all pairs $(m,d)$ of integers such that there is an endomorphism of
$E$ of degree $m$ and with kernel containing exactly $d$ points of order~$2$, 
for $m\le 62$. We compute the analogous set $S_F$ corresponding to homomorphisms 
$E\to F$. Then, for every $n$ from $2$ to $31$, we check to see whether we can 
find an $(m_1,d_1)\in S_E$ and an $(m_2,d_2)\in S_F$ with $d_1 = d_2$ and with
$m_1 + m_2 = 2n$.

For all of the pairs $(j_E,j_F)$ that meet this requirement, we output the
discriminants of the endomorphism rings of $E$ and $F$, and we note whether
$j_E = j_F$. (We can compute the discriminant of $\End F$ by finding the 
discriminant whose Hilbert class polynomial is equal to the minimal polynomial
of~$j_F$.) The computation gives us the list of discriminant pairs listed in
the proposition, and tells us whether $E\cong F$ when the discriminants are
equal.
\end{proof}

\begin{rem}
Magma includes many of the classical modular polynomials $\Psi_m$ for $m<62$ in
its standard distribution, but not all of them. It does include those for prime 
powers less than $61$. We obtained $\Psi_{61}$ from 
\href{https://math.mit.edu/~drew/}{Andrew Sutherland's web page}; it was
calculated using the methods of~\cite{BrokerLauterSutherland2012}. For $m$ with
more than one prime factor, we write $m = ab$ for coprime $a$ and~$b$, and note
that $\Psi(x,y)$ can be computed by taking the $z$-resultant of $\Psi(x,z)$ and
$\Psi(y,z)$.
\end{rem}

\begin{rem}
\label{R:59}
The quadratic order of discriminant $-59$ has class number~$3$, so there are 
three elliptic curves with this endomorphism ring; call them $E$, $F$, 
and~$F'$. Let $S_E$ and $S_F$ be as in the penultimate paragraph of the proof of
Proposition~\ref{P:Deltas}. We find that for every $n$ from $2$ to $31$, there
is an $(m_1,d_1)\in S_E$ and an $(m_2,d_2)\in S_F$ with $d_1 = d_2$ and with
$m_1 + m_2 = 2n$, so $E$ and $F$ meet the ``weaker condition'' mentioned in the
second paragraph of the proof. However, as we saw in the proof of 
Proposition~\ref{P:Deltap}, there do not exist $\alpha_3,\alpha_4\in \End E$
and $\beta_3, \beta_4\in \Hom(E,F)$ with $6 = \degree\alpha_3 +\degree\beta_3$
and $8 = \degree\alpha_4 +\degree\beta_4$ and with~\eqref{EQ:antiisometry} 
holding for both pairs. It is for this reason that we made a special argument
to exclude $p=5$, $\Delta = -59$ from the conclusion of
Proposition~\ref{P:Deltap}.
\end{rem}

Let $\CO$ be an imaginary quadratic order corresponding to one of the 
discriminants listed in Proposition~\ref{P:Deltas}, and let $\CO_\maxl$ be the 
maximal order containing it. We check that the class number of $\CO$ is at
most~$2$, and that if $\CO_\maxl \ne \CO$ then the class number of $\CO_\maxl$
is~$1$. This makes it particularly easy to find the elements of the fundamental
domains $\CF_1$ and $\CF_2$ that correspond to elliptic curves with one of these
endomorphism rings. Namely, if an order $\CO = \BZ[\theta]$ has class 
number~$1$, then the lattice $\langle 1,\theta\rangle$ has CM by $\CO$, and the 
image $\vartheta$ of $\theta$ in the upper half-plane gives rise to the unique 
elliptic curve over $\BC$ with CM by $\CO$. It is a simple matter to find the 
element $\tau$ of $\CF_1$ that is in the $\PSL_2(\BZ)$-orbit of $\vartheta$. The
elements of $\CF_2$ that correspond to the elliptic curve with CM by $\CO$ are
simply the images of $\tau$ under the M\"obius transformations listed in 
Figure~\ref{F:X2}.

If $\CO = \BZ[\theta]$ has class number~$2$, then in addition to the values of
$\tau$ in $\CF_1$ and $\CF_2$ specified above, we must find the $\tau$ 
corresponding to a nonprincipal ideal of~$\CO$. If $\CO$ is maximal, we find a 
nonprincipal prime ideal $I$ of $\CO$ lying over a rational prime~$p$ and we 
write $I = \langle p,\gamma\rangle$ for some $\gamma\in\CO$. Then we let 
$\vartheta$ be the image of $\gamma/p$ in the upper half-plane, and in the same 
way as above we find the elements $\tau$ of $\CF_1$ and $\CF_2$ that lie in the 
$\PSL_2(\BZ)$-orbit of $\vartheta$.

If $\CO$ has class number~$2$ and is nonmaximal, then as we noted above its
corresponding maximal order $\CO_\maxl$ has class number~$1$. Let $f$ be the
conductor of~$\CO$. Then the class group of $\CO$ is isomorphic to the quotient
of the group $U\colonequals (\CO_\maxl / f \CO_\maxl)^\times/(\BZ / f\BZ)^\times$
by the image of $\CO_\maxl^\times$ in~$U$ 
(see~\cite[Eq~(7.25), p.~116]{Cox2022}). If $\gamma\in\CO_\maxl$ represents
a nontrivial element of this quotient group, then we let $\vartheta$ be the 
image of $\gamma/f$ in the upper half-plane, and proceed as before.

We are now in a position to compute the values of $\tau$ and $\sigma$ such that
the period lattice~\eqref{EQ:periodmatrix} with 
polarization~\eqref{EQ:polarization} corresponds to the Jacobian of a curve with
maps of every degree $n>1$ to the elliptic curve corresponding to the element
$\tau$ of~$\CF_1$.

\begin{proposition}
\label{P:sigmatau}
The pairs $(\tau,\sigma)$ listed in Table~\textup{\ref{table:data}} are exactly
the elements of $\CF_1\times \CF_2$ such that the period 
matrix~\eqref{EQ:periodmatrix} and polarization~\eqref{EQ:polarization} 
correspond to a curve from Theorem~\textup{\ref{T:main}} whose associated 
elliptic curve corresponds to $\tau\in\CF_1$.
\end{proposition}

\begin{proof}
Our proof is computational, and Magma programs for carrying out the computation
are available at \url{https://github.com/everetthowe/many-maps}.

For every pair $(\Delta_E, \Delta_F)$ in Proposition~\ref{P:Deltas}, we let
$K = \BQ(\sqrt{\Delta_E}) \cong \BQ(\sqrt{\Delta_F})$ and we specify an
embedding $K$ into $\BC$ by choosing one of the square roots of $\Delta_E$ in
$K$ and declaring that it has positive imaginary part. Then we compute the 
values of $\tau$ in $K \cap \CF_1$ corresponding to $\Delta_E$ as described
above, and the values of $\rho \in K\cap \CF_1$ corresponding to $\Delta_F$. For
each $\rho$ we compute its images $\sigma$ in $K\cap \CF_2$; but if 
$\Delta_E=\Delta_F$ and $-\Delta_E\in\{3,4,7,8,11,12,16,19,20\}$ we only do so
when $\rho=\tau$, and if $\Delta_E = \Delta_F$ and $-\Delta_E\in\{24,36\}$ we 
only do so when $\rho\ne\tau$.

For each pair $(\tau, \sigma)$ we obtain, we perform the following calculation.
Let $\Lambda$ be the lattice in $K^2\subset \BC^2$ generated by the vectors 
\[
b_1 \colonequals (     1,        0)\,, \qquad
b_2 \colonequals (     0,        1)\,, \qquad
b_3 \colonequals (\tau/2,      1/2)\,, \qquad
b_4\colonequals  (   1/2, \sigma/2)\,.
\]
We write elements of $\Lambda$ as $\BZ$-linear combinations of these vectors.

Let $\Lambda_E$ be the lattice in $K\subset \BC$ generated by $1$ and $\tau$. 
If $\Lambda$, with its principal polarization~\eqref{EQ:polarization}, is the
Jacobian of a genus-$2$ curve $C$, then the maps from $C$ to $E$ that take a 
fixed base point to the origin of $E$ correspond to embeddings of $\Lambda_E$
into $\Lambda$. Such an embedding is determined by where it sends 
$1\in \Lambda_E$, and the image $x\in\Lambda$ must have the property that
$\tau x\in\Lambda$; in other words, $x$ must be an element of the sublattice
$\Mu\colonequals \Lambda \cap \tau^{-1}\Lambda$ of $\Lambda$.

If $\Lambda$ is the Jacobian of a curve $C$, then the degree of the map from $C$ 
to $E$ corresponding to an element $x$ of $\Mu$ is equal to the value of the
pairing~\eqref{EQ:polarization} applied to $\tau x$ and $x$. Write 
$\langle\tau x, x\rangle$ for the value of this pairing. If we compute four
elements $c_1, c_2, c_3, c_4$ of $\Lambda$ that generate~$\Mu$, we can then 
compute the Gram matrix of the quadratic form $q$ on $\BZ^4$ that sends a vector
$(n_1,n_2,n_3,n_4)$ to $\langle\tau x, x\rangle$, where 
$x = n_1 c_1 +\cdots + n_4 c_4$.

Given the Gram matrix of $q$, we can compute all of $v\in\BZ^4$ with 
$q(v)\le 31$. Let $S$ be this set of vectors, and let $V$ be the set 
$\{q(v) : v \in S\}$. If $1\in V$ then we know that $\Lambda$ is not the
Jacobian of a genus-$2$ curve $C$, because there can be no degree-$1$ map from
a genus-$2$ curve to $E$. On the other hand, if $V$ does not contain $1$, then
$\Lambda$ does correspond to a curve $C$, and if $V$ does not contain every
integer between $2$ and $31$ then $C$ certainly does not have maps of every 
degree to $E$.

Thus, we can remove from consideration every pair $(\tau, \sigma)$ for which the
set $V$ is not equal to $\{2,\ldots,31\}$.

On the other hand, for every pair $(\tau, \sigma)$ for which the set $V$ is
equal to $\{2,\ldots,31\}$, we can check to see whether the quadratic form $q$
is isomorphic to one of the forms $q_1$, $q_2$, $q_3$, or $q_4$. When we perform
this calculation, we find that in fact every such $q$ \emph{is} isomorphic to
one of the forms~$q_i$. So for each such $q$, we output the values 
$\Delta_E$, $\Delta_F$, $\tau$, $\sigma$, and $i$.

When we do so, we find that the output matches Table~\ref{table:data}. This
proves Theorem~\ref{T:main}.
\end{proof}

\section{Models of the curves}
\label{S:models}
Given values $\Delta_E$, $\Delta_F$, $\tau$, and $\sigma$ from a row of
Table~\ref{table:data}, let $K$ be a number field in which the Hilbert class
polynomials of $\Delta_E$ and $\Delta_F$ split. We can compute models over $K$
for elliptic curves $E$ and $F$ with period lattices $\Lambda_E$ and $\Lambda_F$
homothetic to $\langle 1, \tau\rangle$ and $\langle 1, \sigma\rangle$, 
respectively. Let $L$ be an extension of $K$ over which the $2$-torsion points
of $E$ and $F$ are rational. By complex approximation we can identify the 
$2$-torsion points $P_1$, $P_\tau$, and $P_{1+\tau}$ of $E(L)$ corresponding to
the values $1/2$, $\tau/2$, and $(1+\tau)/2$ modulo $\Lambda_E$, and similarly
we can compute the analogously-defined $2$-torsion points $Q_1$, $Q_\sigma$, and
$Q_{1+\sigma}$ on~$F(L)$. Let $\psi$ be the isomorphism $E[2]\to F[2]$ that
sends $P_1$ to $Q_\sigma$ and $P_\tau$ to $Q_1$. Then we can use the formulas
from~\cite[\S2]{HoweLeprevostPoonen2000} to compute a curve $C$ over $L$ that
corresponds to $E$, $F$, and $\psi$ as in Proposition~\ref{P:structure}.

Once we have a curve $C$ in hand, we can try to find a twist of it that is
rational over its field of moduli and that has a relatively simple defining
equation. (``Relatively simple'' is an inexact expression, so creating these 
models is not an exact science.)

By this method, we have found the models for the curves in Theorem~\ref{T:main}
that we present in Table~\ref{table:equations}. Since we define the curves in 
terms of elements of abstract number fields, and not by specific complex
numbers, each of the equations corresponds to several curves from
Theorem~\ref{T:main} --- namely, the ones with the same values of $\Delta_E$
and~$\Delta_F$.

\begin{table}[t]
\begin{center}
\caption{For each pair ($\Delta_E$, $\Delta_F$), we give a polynomial $f$ such
that $y^2 = f$ is an equation for the corresponding curves $C$ from 
Table~\ref{table:data}}
\begin{tabular}{rrl}
\toprule
$\Delta_E$ & $\Delta_F$ &  Polynomial $f$\\
\midrule
$-4$  & $-100$ & $x^6 - 3x^4 + (2 + r^{24})x^2 - r^{24}$                              \\
      &        & where $r^2 - r - 1 = 0$                                              \\[1ex]
$-8$  & $-32$  & $x^6 + (32r^3 - 31r^2 + 8r - 18)x^4 + (8r^3 - 8r^2 + 16r)x^2 + 8$    \\
      &        & where $r^4 - 2r^2 - 1 = 0$                                           \\[1ex]
$-8$  & $-72$  & $x^6 + (-2r - 33)x^4 + (-116r - 189)x^2 - 2r + 5$                    \\
      &        & where $r^2-6 = 0$                                                    \\[1ex]
$-12$ & $-3$   & $x^6 + (3r - 6)x^4 + (-12r + 9)x^2 + 4$                              \\
      &        & where $r^2 + 1 = 0$                                                  \\[1ex]
$-16$ & $-4$   & $x^6 + (6r + 9)x^4 + (72r - 30)x^2 + 16r$                            \\
      &        & where $r^2-2 = 0$                                                    \\[1ex]
$-20$ & $-20$  & $x^5 + 5x^3 + 5x$                                                    \\[1ex]
$-24$ & $-24$  & $x^6 - 21x^4 + 48x^3 - 45x^2 + 48x - 23$                             \\[1ex]
$-36$ & $-36$  & $x^5 - (8r - 12)x^4 - (73r + 6)x^3 - (168r + 252)x^2 - (72r + 423)x$ \\
      &        & where $r^2 + 3 = 0$                                                  \\
\bottomrule
\end{tabular}
\label{table:equations}
\end{center}
\end{table}

\begin{example}
For the first curve on the list, we will present a basis for the rank-$4$
$\BZ$-module of maps from $C$ to $E$ that take the point $(1,0)$ to the origin
of~$E$. We give the reasonably simple formulas here; Magma code that
verifies~\eqref{EQ:example} can be found in the GitHub repository mentioned in
Section~\ref{S:enumerating}.

Let $i$ and $s$ satisfy $i^2 = -1$ and $s^2 = 5$, and let $r = (s+1)/2$, so that
$r^2 - r - 1= 0$. Our curve $C$ is
\[
y^2 = x^6 - 3x^4 + (2 + r^{24})x^2 - r^{24},
\]
and we let $E$ be the elliptic curve
\[
w^2 = z^3 + 9sz
\]
with $j$-invariant $1728$.

Define rational functions $P_2$, $Q_2$, $P_3$, and $Q_3$ by
\begin{align*}
P_2 & \colonequals \frac{36 s r^6}{x^2-1}                         &  
Q_2 & \colonequals \frac{-18 s r^3}{(x^2-1)^2}                    \\
P_3 & \colonequals \frac {-3 r (x + 1)  (x^2 - 6r^3 x + r^{12})}
                          {(x - 1)(sx + r^6)^2}                   &
Q_3 & \colonequals \frac{ -9r  ((1-2s)x^2 - 2x - r^6)}
                         {(x - 1)^2  (sx + r^6)^3}
\end{align*}
and define maps from $C$ to $E$ by
\begin{align*}
\varphi_1\colon \quad (x,y) &\to (\phantom{-}P_2, \phantom{i} y Q_2)\\
\varphi_2\colon \quad (x,y) &\to (          -P_2,           i y Q_2)\\
\varphi_3\colon \quad (x,y) &\to (\phantom{-}P_3,           i y Q_3)\\
\varphi_4\colon \quad (x,y) &\to (          -P_3, \phantom{i} y Q_3).
\end{align*}
These maps all send the point $(1,0)$ on $C$ to the identity of $E$.
Then one can check that for integers $a,b,c,d$, we have
\begin{equation}
\label{EQ:example}
\degree (a\varphi_1 + b\varphi_2 + c\varphi_3 + d\varphi_4)
= 2a^2 + 2b^2 + 3c^2 + 3d^2 + 2ad + 2bc 
= q_2(a,b,c,d)\,,
\end{equation}
where $q_2$ is the quadratic form given in the introduction.
\end{example}

\begin{rem}
The code in our GitHub repository also includes similar presentations of the
curves and maps for the cases $\Delta_E = \Delta_F = -20$ and 
$\Delta_E = \Delta_F = -36$. We hope to add more examples as time allows.
\end{rem}

\begin{rem}
The examples with $\Delta_E = \Delta_F$ (rows 13 through 20 in
Table~\ref{table:data}) have a remarkable property: Each $C$ has maps of every
degree to \emph{two different} (but Galois conjugate) elliptic curves. If the
existence of any $(C,E)$ pairs as in Theorem~\ref{T:main} is surprising, then
surely it is even more surprising to find curves $C$ with more than one choice
for~$E$!
\end{rem}

\section{Quaternary quadratic forms representing \texorpdfstring{\\}{}
         all integers greater than \texorpdfstring{$1$}{1}}
\label{S:forms}

In this section we will prove Proposition~\ref{P:forms}. It is easy to check
that each of the four forms $q_1$, $q_2$, $q_3$, $q_4$ represents the 
integer~$4$, and it is clearly the case that if a form represents $n$ then it
also represents~$4n$.  Thus, it will suffice for us to show that each of the
four forms represents every integer $n>1$ that is not a multiple of~$4$. We
give a separate argument for each of the four forms.

\subsection{The quadratic form \texorpdfstring{$q_1$}{q1}}
Recall that 
\[ 
q_1 = 2w^2 + 3x^2 + 3y^2 + 4z^2 + 2xy\,. 
\]
Suppose $n>1$ is not a multiple of~$4$. Let $d$ be the integer defined by
\[
d = \begin{cases}
     0 & \text{if $n\equiv 2, 3, 5\bmod 8$}\\
     1 & \text{if $n\equiv 1, 6, 7\bmod 8$\,.}
    \end{cases}
\]
Then $n-4d^2$ is a positive integer and $n-4d^2\equiv 2, 3, 5\bmod 8$; a result
of Dickson~\cite[Theorem~VI]{Dickson1927} then shows that we may write 
\[
n - 4d^2 = a^2 + 2b^2 + 2c^2
\]
for some integers $a,b,c$. By considering the right-hand side of this equality 
modulo~$8$, we check that $b$ and $c$ cannot both have the opposite parity
to~$a$, so by switching $b$ and $c$ if necessary we can ensure that  
$a\equiv b\bmod 2$. Now we simply set
\begin{align*}
w &= c                  & y &= (b-a)/2  \\
x &= (a+b)/2            & z &= d
\end{align*}
and note that $n = q_1(w,x,y,z)$.

\subsection{The quadratic form \texorpdfstring{$q_2$}{q2}}
Recall that 
\[
q_2 = 2w^2 + 2x^2 + 3y^2 + 3z^2 + 2wz + 2xy\,.
\]
Suppose $n>1$ is not a multiple of~$4$. Let $d$ be the integer defined by
\[
d = \begin{cases}
     0 & \text{if $3n\equiv 1, 2, 5, 6, 7\bmod 8$}\\
     1 & \text{if $3n\equiv 3\bmod 8$\,.}
\end{cases}
\]
Then $3n-5d^2$ is a positive integer and $3n-5d^2\equiv 1, 2, 5, 6, 7\bmod 8$.
Another result of Dickson~\cite[Theorem~15]{Dickson1926} shows that we may write 
\[
3n - 5d^2 = a^2 + b^2 + 5c^2
\]
for some integers $a,b,c$. It follows that 
\[
a^2 + b^2 \equiv c^2 + d^2 \bmod 3\,,
\]
so by exchanging $a$ and $b$, if necessary, we may assume that 
$a^2\equiv c^2\bmod 3$ and $b^2\equiv d^2\bmod 3$, and by replacing $a$ and $b$
with their negatives, if necessary, we may assume that in fact 
$a\equiv c\bmod 3$ and $b\equiv d\bmod 3$. Now let
\begin{align*}
w &= c  & y &= (b-d)/3\\
x &= d  & z &= (a-c)/3
\end{align*}
and note that $n = q_2(w,x,y,z)$.

\subsection{The quadratic form \texorpdfstring{$q_3$}{q3}}
Recall that 
\[
q_3 = 2w^2 + 3x^2 + 3y^2 + 4z^2 + 2wx + 2wy + 2xz + 2yz\,.
\]
Let $n>1$ be an integer that is not a multiple of~$4$. Let $d$ be the integer 
defined by
\[
d = \begin{cases}
     0 & \text{if $n\equiv 2, 3, 6, 7\bmod 8$}\\
     1 & \text{if $n\equiv 1, 5\bmod 8$\,.}
\end{cases}
\]
Then $n-3d^2$ is a positive integer and $n-3d^2\equiv 2, 3, 6, 7\bmod 8$. Yet
another result of Dickson~\cite[Theorem~XI]{Dickson1927} shows that we may write
\[
n - 3d^2 = a^2 + 2(b^2 + bc + c^2)
\]
for some integers $a, b, c$.  Now, $b$ and $c$ cannot both be even, because in 
that case we would have $n-3d^2 \equiv a^2 \bmod 8$, while we know that $n-3d^2$
is not congruent to a square modulo~$8$. By replacing $(b,c)$ with $(b+c,-c)$ if
necessary, we can ensure that one of the numbers $b,c$ is even and the other 
odd. Then by switching $b$ and~$c$, if
necessary, we can ensure that $a + c + d$ is even. Now set
\begin{align*}
w &= b + (a+c+d)/2 & y &= c - (a+c+d)/2\\
x &=   - (a+c+d)/2 & z &= d           
\end{align*}
and note that $n = q_2(w,x,y,z)$.

\subsection{The quadratic form \texorpdfstring{$q_4$}{q4}}
Recall that 
\[
q_4 = 2w^2 + 3x^2 + 4y^2 + 6z^2 - 2wx + 2wz + 2xy + 4yz\,.
\]
Our proof in this case depends on the parity of $n$. First let us suppose that
$n > 1$ is an even integer that is not a multiple of~$4$. Then $n$ is
congruent to $2$ or $6$ modulo~$8$, so Legendre's three-square 
theorem~\cite[pp.~398--399]{Legendre1798} says that we may write
\[
n = a^2 + b^2 + c^2
\]
for some integers $a,b,c$. By permuting these integers and changing their signs,
if necessary, we may assume that $a\equiv b\bmod 3$. Then we may set
\begin{align*}
w &= (2a + b)/3     & y &= (a - b + 3c)/6 \\
x &=  0             & z &= (b-a)/3\,.           
\end{align*}
These numbers are integers; the only thing that may not be clear immediately is 
whether $y$ is integral at $2$, but that can be verified by noting that
\[
a - b + 3c \equiv a + b + c 
           \equiv a^2 + b^2 + c^2  
           \equiv n
           \equiv 0\bmod 2\,.
\]
One can easily check that $n = q_4(w,x,y,z).$

Now suppose that $n>1$ is odd. Let $d=3$, and note that the three-square theorem
shows that we may write
\[
4n - d^2 = a^2 + b^2 + c^2
\]
for some integers $a,b,c$. Considering this equality modulo~$4$, we see that 
$a$, $b$, and $c$ must all be odd, and by permuting them and changing their
signs (if necessary) we can assume that $a\equiv b\bmod 3$ and 
$a \equiv b + c + d\bmod 4$.  Now set
\begin{align*}
w &= (2a + b + d)/6 & y &= (a - b + 3c - d)/12 \\
x &= d/3            & z &= (b-a)/6\,.            
\end{align*}
Our assumptions on $a$, $b$, $c$, and $d$ show that $w,x,y,z$ are integers, and
it is easy to check that $n = q_4(w,x,y,z).$

This proves Proposition~\ref{P:forms}.\qed

\section{Empty intersections of Humbert surfaces}
\label{S:intersection}

For a brief introduction to the basics of Humbert surfaces $H_N$ in modern 
terms, see~\cite{Kani1994}. For our purposes, all we need to know is that a
point in the moduli space $\CM_2$ of genus-$2$ curves lies on the Humbert
surface $H_{n^2}$ if and only if the curve it represents has a minimal map of 
degree~$n$ to an elliptic curve. Since a map $C\to E$ of prime degree is 
necessarily minimal, Theorem~\ref{T:main} shows that the intersection 
$\cap_{p} H_{p^2}$ is nonempty, where $p$ ranges over the set of all primes.

This leads to the motivation behind Theorems~\ref{T:intersection1} 
and~\ref{T:intersection2}: Can we show that there are finite collections of
Humbert surfaces $H_{n^2}$ with trivial intersection? And, since the Humbert
surfaces have codimension $1$ and live in a $3$-dimensional moduli space, can we
find \emph{four} Humbert surfaces $H_{n^2}$ whose intersection is empty, as one 
would expect if we were dealing with a random collection of surfaces? Work of
Kani sheds light on the situation, and guides our proofs of these two theorems.

Kani~\cite[Theorem~20]{Kani2019} shows that to every genus-$2$ curve $C$ one can
associate a positive definite quadratic form $q_C$ in at most $3$ variables,
with integer coefficients, known as the \emph{refined Humbert invariant}
of~$C$. The invariant $q_C$ has the property that $C$ has a minimal map of 
degree $n$ to some elliptic curve if and only if $q_C$ represents~$n^2$ 
primitively; that is, if and only if there is an integer vector $v$ whose
entries generate the unit ideal and with $q_C(v) = n^2$. Kani also 
shows~\cite{Kani2025preprint} that a positive-definite ternary quadratic form
with integer coefficients that is \emph{primitive} --- that is, one whose 
coefficients generate the unit ideal --- occurs as refined Humbert invariant if
and only if it does not take on any values that are $2$ or $3$ modulo~$4$. 
Theorems~\ref{T:intersection1} and~\ref{T:intersection2} are therefore
corollaries of the following proposition concerning ternary quadratic forms.

\begin{proposition}
\label{P:ternary}
Let $q$ be a primitive positive definite ternary quadratic form with integer
coefficients that does not represent~$1$ and that does not represent any integer
congruent to $2$ or $3$ modulo~$4$. Then\textup{:}
\begin{enumerate}
\item \label{T1}
      For some $n$ in the set 
      $D_1\colonequals\{2,3,4,5,6,7,8,9,11,12,13,19,31,59\}$, the form $q$ does
      not primitively represent~$n^2$.
\item \label{T2}
      Let $N$ be as in Theorem~\textup{\ref{T:intersection2}} and let $k$ be a 
      positive integer. For some $n$ in the set $D_2\colonequals\{2,3,4,kN\}$, 
      the form $q$ does not primitively represent~$n^2$.
\end{enumerate}
\end{proposition}

\begin{proof}
Let $D$ be either of the sets $D_1$ and $D_2$. Suppose there were a positive
definite integer ternary form $q$ that does not represent $1$, and that does not
represent any integer that is $2$ or $3$ modulo~$4$, but that primitively
represents all $n^2$ with $n\in D$. In particular, this $q$ must primitively
represent $4$, $9$, and $16$. We write $q$ in Minkowski-reduced form as
$ax^2 + by^2 + cz^2 + 2rxy + 2sxz + 2tyz$. (Note that a priori, $r$, $s$, and
$t$ may be half-integers.) 

The condition that $q$ only represent integers that are $0$ and $1$ mod~$4$ 
implies that $a$, $b$, and $c$ must be $0$ or $1$ mod~$4$. By considering the
values of $q$ at triples $(x,y,z)$ with all entries in $\{-1,0,1\}$ we find that
the same congruence condition implies that the half-integers $r$, $s$, and $t$ 
are in fact integers.

By replacing some of the variables with their negations, we can assume that
$r\ge 0$ and $s\ge 0$, and if $rs=0$ then we may assume that $t\ge0$. And 
finally, since the form is reduced, from~\cite[Lemma~1.2, p.~257]{Cassels1978}
we see that 
\begin{align*}
2r &\le a      & 2t &\le a+b+2r-2s        &  2t &\ge -b        \\
2s &\le a      & 2t &\le a + b - 2r + 2s  &  2t &\ge 2r+2s-a-b\,.\\
2t &\le b     
\end{align*}
  
We know from~\cite[Satz~7, p.~281]{vanderWaerden1956} that the sequence
$(a,b,c)$ is the sequence of successive minima for~$q$. Since $q$ represents $4$
but does not represent~$1$, we must have $a=4$. Since $4x^2$ does not 
represent~$9$, we must have $b\in\{4, 5, 8, 9\}$.

We check that for $b\in\{4, 5, 8, 9\}$ and $r\in\{0,1,2\}$, if the quadratic
form $ax^2 + by^2 + 2rxy$ does not represent integers that are $2$ or $3$
mod~$4$ then it also does not primitively represent $16$. Since we are assuming
that $q$ does primitively represent $16$, the third succesive minimum of $q$
must be at most $16$; in other words, $c\le 16$.

Now we can enumerate all of the values of $a$, $b$, $c$, $r$, $s$, and $t$ 
meeting the conditions above. For each resulting ternary form~$q$, we check
whether it represents all of the integers in $D_1$ and $D_2$ as follows.

The elements of $D_1$ are small enough that we can simply check by brute force
whether $q$ primitively represents all of them. Magma code in our GitHub
repository carries out this computation, and verifies that none of the $q$ we
must consider represents every element of $D_1$. This proves
statement~\eqref{T1}.

To prove statement~\eqref{T2}, we must show that none of the $q$ produced above
primitively represents $(kN)^2$. To accomplish this, it will suffice for us to
produce, for each of the $q$ we must consider, a prime power $m$ that divides
$N^2$ such that $q$ does not primitively represent any multiple of~$m$.

Lemma~\ref{L:noprimitive}, below, gives a quick method that often can find such
a prime power~$m$. We apply the lemma to each of the $q$ we must consider. If 
the lemma does not provide us with such an~$m$, the brute force method sketched 
in Remark~\ref{R:p=2} produces an $m$ that divides $2^{10}$ such that $q$ does
not primitively represent any multiple of~$m$. Since $2^{10}$ divides $N^2$, 
this computation verifies statement~\eqref{T2} of the proposition.
\end{proof}

\begin{lemma}[Compare to \protect{\cite[Prop. 3.1]{EarnestGunawardana2021}}]
\label{L:noprimitive}
Let $q = ax^2 + by^2 + cz^2 + 2rxy + 2sxz + 2tyz$ be a nondegenerate quadratic
form with integer coefficients, let $M$ be the matrix
\[
\begin{pmatrix}
a & r & s \\
r & b & t \\
s & t & c
\end{pmatrix},
\]
and for $1\le i\le 3$ let $M_i$ be the upper left $i\times i$ submatrix of~$M$.
Let $A = \det M_1$, $B = (\det M_2)/(\det M_1)$, and
$C = (\det M_3)/(\det M_2)$. Suppose that $p$ is an odd prime such that $A$,
$B$, $C$, $r/a$, $s/a$, and $(rs-at)/(ab-r^2)$ are all integral at~$p$, and 
suppose that the parities of the $p$-adic valuations of $A$, $B$, and $C$ are
not all equal. Finally, suppose that $-A/B$, $-A/C$, and $-B/C$ are all
nonsquares in $\BQ_p$.

If $m$ is an integer whose $p$-adic valuation is greater than that of $A$, $B$,
and $C$, then $q$ does not primitively represent $m$.
\end{lemma}

\begin{proof}
We note that if we set 
\[ u \colonequals x + \Bigl(\frac{r}{a}\Bigr) y + \Bigl(\frac{s}{a}\Bigr)z\,,
   \qquad 
   v \colonequals y - \Bigl(\frac{rs-at}{ab-r^2}\Bigr) z\,,
   \qquad 
   w \colonequals z\,,\]
then $q = Au^2 + Bv^2 + Cw^2$, and our assumptions on the integrality of $A$, 
$B$, $C$, $r/a$, $s/a$, and $(rs-at)/(ab-r^2)$ show that $q$ is in fact
isomorphic over $\BZ_p$ to the diagonal form 
$q'\colonequals Au^2 + Bv^2 + Cw^2$. Therefore it suffices to show that $q'$
does not primitively represent~$m$.

Reorder $A$, $B$, and $C$ so that $A$ and $B$ have $p$-adic valuations of the
same parity. Suppose that $q'$ represents $m$ over $\BZ_p$, and let $u$, $v$,
and $w$ be elements of $\BZ_p$ with $m = Au^2 + Bv^2 + Cw^2$. 

We claim that each summand in this last equality has valuation at least that
of~$m$. For suppose not. Because our valuation is non-Archimedean, if any of the
summands has valuation less than that of~$m$, then at least two of them must
have that same valuation, in order for there to be cancellation that will
increase the valuation of the sum. Since $C$ has different parity of valuation
than $A$ and~$B$, it must be that $Au^2$ and $Bv^2$ have equal valuation, say
equal to $e$, and that $Au^2 + Bv^2$ has valuation greater than~$e$. The
$p$-adic expansions of the summands must then look like
\begin{align}
\notag    Au^2 &= \phantom{-}\alpha p^e + \text{higher-order terms}\\
\notag    Bv^2 &=           -\alpha p^e + \text{higher-order terms}\\
\intertext{with $\alpha$ a unit. But then}
\label{EQ:cong} -\Bigl(\frac{Au^2}{Bv^2}\Bigr) &= 1 + \text{higher-order terms},
\end{align}
which is enough to show that $-(Au^2)/(Bv^2)$ is a square in $\BZ_p$, because
$p$ is odd. This contradicts our assumption that $-A/B$ is not a square.
Therefore, $Au^2$, $Bv^2$, and $Cw^2$ each have valuation greater than that
of~$m$, so each of $u$, $v$, and $w$ has positive valuation, so they are not
coprime.
\end{proof}

\begin{rem}
\label{R:p=2}
The requirement in Lemma~\ref{L:noprimitive} that $p$ be odd is necessary in 
order for~\eqref{EQ:cong} to contradict the assumption that $-A/B$ is a
nonsquare. Rather than try to spell out a variation of the lemma for $p=2$, we
instead simply note that for small powers $m$ of $2$, we can check by brute
force whether we can have $q(x,y,z)\equiv 0 \bmod m$ for values of $x$, $y$, 
and~$z$ that are not all even. (The amount of brute force needed can be reduced
slightly by diagonalizing $q$ over $\BQ_2$; see our Magma code for details.)
\end{rem}

\begin{rem}
\label{R:empty}
Kir~\cite[\S6.4]{Kir2024} gives other examples of finite collections of Humbert
surfaces with trivial intersections. What is new in 
Theorems~\ref{T:intersection1} and~\ref{T:intersection2} is the restriction to
Humbert surfaces with square discriminants, and the fact that only four surfaces
are needed in Theorem~\ref{T:intersection1}.
\end{rem}

We close by presenting the refined Humbert invariants for the curves in
Theorem~\ref{T:main}. Harun Kir pointed out to us that there are at least three
ways to compute the invariant for our curves: 
\begin{enumerate}
\item By using older work of Kani~\cite[Prop.~5.4]{Kani1994}. This result 
      requires knowing the period matrices of the curves, which we have from 
      Table~\ref{table:data}. (Note that there is a sign error 
      in~\cite[Cor.~5.5]{Kani1994}: The coefficient of $d$ should be negated.)
\item By using work of Kir~\cite[Prop.~5]{Kir2026}. This requires writing 
      $\Jac C$ as a product $E_1\times E_2$ of two elliptic curves, and 
      expressing the canonical polarization as an isogeny 
      $E_1\times E_2\to E_1\times E_2$. This is possible for our examples
      because every abelian surface that is \emph{isogenous} to a product of two
      elliptic curves with CM by orders in the same quadratic field is in fact 
      \emph{isomorphic} to a product of two curves with the same property. (This
      follows immediately from a result of Borevi\v{c} and Fadeev on the 
      structure of modules over orders in quadratic 
      fields~\cite{BorevicFaddeev1960}; cf.~\cite{BorevicFaddeev1965},
      \cite{ShiodaMitani1974}, \cite{Kani2011}.)
\item By generalizing more recent work of Kani~\cite{Kani2019} to CM elliptic 
      curves.
\end{enumerate}

We implemented the first method in Magma for the case in which the period matrix
has the form 
\[
\begin{pmatrix}
1 & 0 & \tau_1 & \tau_2\\
0 & 1 & \tau_2 & \tau_3
\end{pmatrix},
\]
where the $\tau_i$ all lie in a single quadratic field; the code can be found in
the file \texttt{Examples.magma} in the GitHub repository mentioned in
Section~\ref{S:enumerating}. Using this code, we find that the refined Humbert
invariants of our examples are as presented in Table~\ref{table:rHi}.

\begin{table}[t]
\caption{For each pair $(\Delta_E$, $\Delta_F)$, we give the refined Humbert
invariant for the corresponding Galois orbit of curves from 
Table~\ref{table:data}}
\begin{center}
\begin{tabular}{rrll}
\toprule
$\Delta_E$ & $\Delta_F$ &&  Refined Humbert invariant\\
\midrule
$-4$  & $-100$ && $4x^2 + 5y^2 + 20z^2$              \\
$-8$  & $-32$  && $4x^2 + 9y^2 + 16z^2 + 4xy$        \\
$-8$  & $-72$  && $4x^2 + 9y^2 +  9z^2 + 6yz$        \\
$-12$ & $-3$   && $4x^2 + 8y^2 +  9z^2 + 4xz + 8yz$  \\
$-16$ & $-4$   && $4x^2 + 8y^2 +  9z^2 + 4xz$        \\
$-20$ & $-20$  && $4x^2 + 4y^2 +  5z^2$              \\
$-24$ & $-24$  && $4x^2 + 5y^2 +  5z^2 + 2yz$        \\
$-36$ & $-36$  && $4x^2 + 5y^2 +  8z^2 + 4yz$        \\
\bottomrule
\end{tabular}
\label{table:rHi}
\end{center}
\end{table}

\bibliographystyle{hplaindoi}
\bibliography{manymaps}
\end{document}